\documentclass[11pt]{article}

\usepackage{amsmath,amssymb,amsfonts,amsthm}
\usepackage{amssymb}
\usepackage{graphicx}
\usepackage{tikz}
\usetikzlibrary{automata}
\usetikzlibrary{arrows}
\tikzset{every loop/.style={min distance=10mm,looseness=10}}
\tikzset{every state/.style={minimum size=2mm}}

\newtheorem{theorem}{Theorem}

\newtheorem{conjecture}{Conjecture}

\title{Solving computational problems in the theory of word-representable graphs}

\author{\"Ozg\"ur Akg\"un\thanks{School of Computer Science, University of St Andrews, St Andrews, Fife KY16 9SX, UK. Emails: \{ozgur.akgun,ian.gent\}@st-andrews.ac.uk
}, Ian P. Gent\footnotemark[1], Sergey Kitaev\thanks{School of Computer and Information Sciences, University of Strathclyde, Glasgow, G1 1HX, UK. Email: sergey.kitaev@cis.strath.ac.uk}, Hans Zantema\thanks{Department of Computer Science, Eindhoven University of Technology, P.O. box 513, 5600 MB Eindhoven, The Netherlands. Email: H.Zantema@tue.nl}}

\begin{document}

\maketitle

\begin{abstract}
A simple graph $G=(V,E)$ is word-representable if there exists a word $w$
over the alphabet $V$ such that letters $x$ and $y$ alternate in $w$ iff $xy\in E$. Word-representable graphs generalize several important classes of graphs. A graph is word-representable iff it admits a semi-transitive orientation. We use semi-transitive orientations to enumerate connected non-word-representable graphs up to the size of 11 vertices, which led to a correction of a published result. Obtaining the enumeration results took 3 CPU years of computation.

Also, a graph is word-representable iff it is $k$-representable for some $k$, that is, if it can be represented using $k$ copies of each letter. The minimum such $k$ for a given graph is called graph's representation number.  Our computational results in this paper not only include distribution of $k$-representable graphs on at most 9 vertices, but also have relevance to a known conjecture on these graphs.  In particular, we find a new graph on 9 vertices with high representation number.

Finally, we introduce the notion of a $k$-semi-transitive orientation refining the notion of a semi-transitive orientation, and show computationally that the refinement is not equivalent to the original definition unlike the equivalence of $k$-representability and word-representability.
\end{abstract}

\section{Introduction}

Letters $x$ and $y$ alternate in a word $w$ if after deleting in $w$ all letters but the copies of $x$ and $y$ we either obtain a word $xyxy\cdots$ (of even or odd length) or a word $yxyx\cdots$ (of even or odd length). For example, the letters 2 and 5 alternate in the word 11245431252, while the letters 2 and 4 do not alternate in this word. A simple graph $G=(V,E)$ is {\em word-representable} if there exists a word $w$ over the alphabet $V$ such that letters $x$ and $y$ alternate in $w$ iff $xy\in E$. By definition, $w$ {\em must} contain {\em each} letter in $V$. We say that $w$ {\em represents} $G$, and that $w$ is a {\em word-representant}. 

The definition of a word-representable graph works both for vertex-labeled and unlabeled graphs because any labeling of a graph $G$ is equivalent to any other labeling of $G$ with respect to word-representability (indeed, the letters of a word $w$ representing $G$ can always be renamed). For example, the graph to the left in Figure~\ref{wrg-ex} is word-representable because its labeled version to the right in Figure~\ref{wrg-ex} can be represented by 1213423. For another example, each {\em complete graph} $K_n$ can be represented by any permutation $\pi$ of $\{1,2,\ldots,n\}$, or by $\pi$ concatenated any number of times.   Also, the {\em empty graph} $E_n$ (also known as {\em edgeless graph}, or {\em null graph}) on vertices $\{1,2,\ldots,n\}$ can be represented by $12\cdots (n-1)nn(n-1)\cdots 21$, or by any other permutation concatenated with the same permutation written in the reverse order.

\begin{figure}[h]
\begin{center}
\begin{tabular}{ccc}
\begin{tikzpicture}[node distance=1cm,auto,main node/.style={fill,circle,draw,inner sep=0pt,minimum size=5pt}]

\node[main node] (1) {};
\node[main node] (2) [below left of=1] {};
\node[main node] (3) [below right of=1] {};
\node[main node] (4) [below right of=2] {};

\path
(1) edge (2)
(1) edge (3);

\path
(2) edge (3)
(2) edge (4);
\end{tikzpicture}

& 

\ \ \ \ \ \ \

&

\begin{tikzpicture}[node distance=1cm,auto,main node/.style={circle,draw,inner sep=1pt,minimum size=2pt}]

\node[main node] (1) {{\tiny 3}};
\node[main node] (2) [below left of=1] {{\tiny 2}};
\node[main node] (3) [below right of=1] {{\tiny 4}};
\node[main node] (4) [below right of=2] {{\tiny 1}};

\path
(1) edge (2)
(1) edge (3);

\path
(2) edge (3)
(2) edge (4);

\end{tikzpicture}

\end{tabular}
\end{center}
\vspace{-5mm}
\caption{An example of a word-representable graph}\label{wrg-ex}
\end{figure}
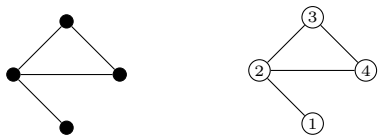

We note that the class of word-representable graphs is {\em hereditary}. That is, removing a vertex $v$ in a word-representable graph $G$ results in a word-representable graph $G'$. Indeed, if $w$ represents $G$ then $w$ with $v$ removed represents $G'$. 

There is a long line of research on word-representable graphs (see, e.g.\ \cite{AKM15,CKS16,CKS16A,CKL17,GKZ16,G16,GK18,HKP10,HKP11,HKP16,K13,KP08,M16}) that is summarized in \cite{K17,KL15}. The roots of the theory of word-representable graphs are in the study of the celebrated {\em Perkins semigroup} in \cite{KS08}, which has played a central role in semigroup theory since 1960, particularly as a source of examples and counterexamples. However, the significance of word-representable graphs is in the fact that they generalize several important classes of graphs such as  $3$-colorable graphs, comparability graphs and circle graphs. 

One of the key tools to study word-representable graphs is the notion of a semi-transitive orientation to be defined next. 

\subsection{Semi-transitive orientations}\label{semi-transitive-sec}

The notion of a semi-transitive orientation was introduced in \cite{HKP11,HKP16}, but we follow \cite[Section 4.1]{KL15} to introduce it here.
A graph $G=(V,E)$ is {\em semi-transitive} if it admits
an {\em acyclic} orientation such that for any directed path 
$v_1\rightarrow v_2\rightarrow \cdots \rightarrow v_k$ with $v_i\in V$ for all $i$, $1\leq i\leq k$, either
\begin{itemize}
\item there is no edge $v_1\rightarrow v_k$, or 
\item the edge $v_1\rightarrow v_k$ is present and there are edges $v_i\rightarrow v_j$ for all $1\le i<j\le k$. 
In other words, in this case, the (acyclic) subgraph induced by the vertices $v_1,\ldots,v_k$ is transitive (with the unique source $v_1$ and the unique sink $v_k$).  
\end{itemize}
We call such an orientation {\em semi-transitive}. In fact, the notion of a semi-transitive orientation is defined in \cite{HKP11,HKP16} in terms of {\em shortcuts} as follows. A {\em semi-cycle} is the directed acyclic
graph obtained by reversing the direction of one edge of a directed cycle in which the directions form a directed path. An acyclic digraph is a shortcut if it is induced by
the vertices of a semi-cycle and contains a pair of non-adjacent
vertices. Thus, a digraph on the vertex set $\{ v_1, \ldots,
v_k\}$ is a shortcut if it contains a directed path $v_1\rightarrow v_2\rightarrow \cdots
\rightarrow v_k$, the edge $v_1\rightarrow v_k$, and it is missing an edge $v_i\rightarrow v_j$ for some $1 \le i
< j \le k$; in particular, we must have $k\geq 4$, so that any shortcut is on at least four vertices. Clearly, this definition is just another way to introduce the notion of a semi-transitive orientation presented above. 

It is not difficult to see that
all transitive (that is, comparability) graphs are semi-transitive,
and thus semi-transitive orientations are a generalization of transitive orientations.  A key theorem in the theory of word-representable graphs is presented next.

\begin{theorem}[\cite{HKP11,HKP16}]\label{key-thm} A graph $G$ is word-representable if and only if it admits a semi-transitive orientation (that is, if and only if $G$ is semi-transitive).
\end{theorem}

A corollary to Theorem~\ref{key-thm} is the useful fact that any $3$-colorable graph is word-representable.

\subsection{Comparability graphs and permutational representation}

An orientation of a graph is {\em transitive} if the presence of edges $u\rightarrow v$ and $v\rightarrow z$ implies the presence of the edge $u\rightarrow z$.  An unoriented graph is a {\em comparability graph} if it admits a transitive orientation. A graph $G=(V,E)$ is {\em permutationally representable} if it can be represented by a word of the form $p_1\cdots p_k$ where $p_i$ is a permutation. 

The following theorem is an easy corollary of the fact that any partially ordered set can be represented as intersection of linear orders, and that a linear order can be represented by a permutation.

\begin{theorem}[\cite{KS08}]\label{Kit-Sei-thm} A graph is permutationally representable if and only if it is a comparability graph.\end{theorem}

Permutational representation of a graph is a special case of uniform representation, to be discussed next.

\subsection{Uniform representations}

A word $w$ is {\em $k$-uniform} if each letter in $w$ occurs $k$ times. For example, the word 342321441231 is 3-uniform, while 43152 is a 1-uniform word (a permutation).  A graph $G$ is {\em $k$-word-representable}, or {\em $k$-representable} for brevity, if there exists a $k$-uniform word $w$ representing it. We say that $w$ {\em $k$-represents}  $G$. A somewhat surprising fact establishes equivalence of word-representability and uniform word-representability:

\begin{theorem}[\cite{KP08}]\label{equiv-thm} A graph is word-representable iff it is $k$-representable for some $k$. \end{theorem}

Thus, in the study of word-representable graphs, word-representants can be assumed to be uniform. {\em Graph's representation number} is the {\em least} $k$ such that the graph is $k$-representable. For non-word-representable graphs, we let $k=\infty$. It is known \cite{HKP16} that the upper bound on a shortest uniform word-representant for a graph $G$ on $n$ vertices is essentially $2n^2$, that is, one needs at most $2n$ copies of each letter to represent $G$. We let $\mathcal{R}(G)$ denote $G$'s representation number  and $\mathcal{R}_k=\{ G : \mathcal{R}(G)=k\}$.

The class of complete graphs is clearly the class of graphs with representation number 1. Further, the class of graphs with representation number $2$ is precisely the class of {\em circle graphs}, that is, the {\em intersection graphs} of sets of chords of a circle \cite{HKP11}. Unlike the cases of graphs with representation numbers 1 or 2, no characterization of graphs with representation number 3, or higher, is known. However, there is a number of interesting results on graphs with representation numbers higher than 2, some of which we mention next (see \cite{K17} for references to the original sources in relation to the results, and for more results in this direction).

The representation number of the {\em Petersen graph} and any {\em prism} is 3. Also, for every graph $G$ there are infinitely many $3$-representable graphs $H$
that contain $G$ as a {\em minor}. Such a graph $H$ can be obtained from $G$ by subdividing {\em each} edge into {\em any} number of, but {\em at least} three edges. 

\subsection{Graphs with high representation number}

As for graphs with high representation number, only {\em crown graphs} and graphs $G_n$ based on them (see the definitions below) were known until this paper; Figure~\ref{9-4} gives an example of another such graph.  A {\em crown graph} (also known as a {\em cocktail party graph}) $H_{n,n}$ is obtained from the complete bipartite graph $K_{n,n}$ by removing a {\em perfect matching}. That is, $H_{n,n}$ is obtained from $K_{n,n}$ by removing  $n$ edges such that {\em each} vertex was incident to {\em exactly one} removed edge. See Figure~\ref{crown-graphs-fig} for examples of crown graphs. 
 
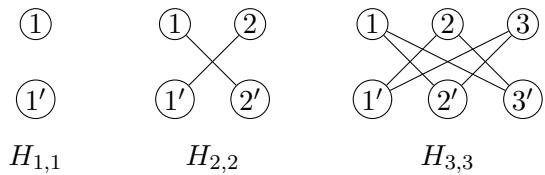
\begin{figure}
\begin{center}

\begin{tabular}{ccccc}

\begin{tikzpicture}[node distance=1cm,auto,main node/.style={circle,draw,inner sep=1pt,minimum size=5pt}]

\node[main node] (1) {1};
\node[main node] (2) [below of=1] {$1'$};
\node (a) [below of=2,yshift=2mm] {$H_{1,1}$};


\end{tikzpicture}

&

\ \ \ 

&

\begin{tikzpicture}[node distance=1cm,auto,main node/.style={circle,draw,inner sep=1pt,minimum size=5pt}]

\node[main node] (1) {1};
\node[main node] (2) [below of=1] {$1'$};
\node[main node] (3) [right of=1] {2};
\node[main node] (4) [below of=3] {$2'$};
\node (a) [below of=2,yshift=2mm,xshift=5mm] {$H_{2,2}$};

\path
(1) edge node  {} (4);

\path
(3) edge node  {} (2);

\end{tikzpicture}

&

\ \ \ 

&

\begin{tikzpicture}[node distance=1cm,auto,main node/.style={circle,draw,inner sep=1pt,minimum size=5pt}]

\node[main node] (1) {1};
\node[main node] (2) [below of=1] {$1'$};
\node[main node] (3) [right of=1] {2};
\node[main node] (4) [below of=3] {$2'$};
\node[main node] (5) [right of=3] {3};
\node[main node] (6) [below of=5] {$3'$};
\node (a) [below of=4,yshift=2mm] {$H_{3,3}$};

\path
(1) edge node  {} (4)
     edge node  {} (6);

\path
(3) edge node  {} (2)
     edge node  {} (6);

\path
(5) edge node  {} (2)
     edge node  {} (4);

\end{tikzpicture}

\vspace{-4mm}

\ 

\end{tabular}

\end{center}
\vspace{-2mm}
\caption{Crown graphs}\label{crown-graphs-fig}
\end{figure}

By Theorem~\ref{Kit-Sei-thm}, $H_{n,n}$ can be represented by a concatenation of permutations, because $H_{n,n}$ is a comparability graph (to see this, just orient all edges from one part to the other). In fact, $H_{n,n}$ is known to require $n$ permutations to be represented (the maximum possible amount for a comparability graph on $2n$ vertices by the well known theorem on the poset dimensions by Hiraguchi~\cite{H51}). However, we can provide a shorter representation for $H_{n,n}$, to be discussed next, which is still long (linear in $n$). 

Note that $H_{1,1}\in\mathcal{R}_2$. Further, $H_{2,2}\neq K_4$, the complete graph on four vertices, and thus $H_{2,2}\in\mathcal{R}_2$ because it cannot be represented by a permutation but can be 2-represented by $121'2'212'1'$. Also, $H_{3,3}=C_6$, a {\em cycle graph}, which belongs to $\mathcal{R}_2$ as is shown, e.g.\ in \cite{K17,KL15}. Finally, $H_{4,4}\in\mathcal{R}_3$ because  $H_{4,4}$ is a prism (it is the $3$-dimensional cube). The following theorem gives the representation number $\mathcal{R}(H_{n,n})$ in the remaining cases.

\begin{theorem}[\cite{G16}]\label{crown-thm-main} If $n\geq 5$, then the representation number of $H_{n,n}$ is $\lceil n/2 \rceil$.\end{theorem}

\begin{conjecture}\label{conj1} $H_{n,n}$ has the highest representation number among all bipartite graphs on $2n$ vertices. \end{conjecture}

The graph $G_n$ is obtained from a crown graph $H_{n,n}$ by adding an apex (all-adjacent vertex). See Figure~\ref{G4-fig} for the graph $G_4$. It turns out that $G_n$ is the {\em worst} known word-representable graph in the sense that it requires the maximum number of copies of each letter to be represented, as recorded in the following theorem.  

\begin{theorem}[\cite{KP08}]\label{G-represent-thm} The representation number of $G_n$ is $\lfloor (2n+1)/2 \rfloor$.\end{theorem}

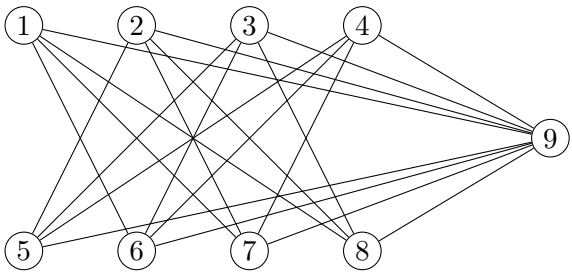
\begin{figure}
\begin{center}
\hspace{10mm} \begin{tikzpicture}
\node[circle,draw,inner sep=0pt,minimum width=5mm] (1) at (0,3) {$1$};
\node[circle,draw,inner sep=0pt,minimum width=5mm] (2) at (1.5,3) {$2$};
\node[circle,draw,inner sep=0pt,minimum width=5mm] (3) at (3,3) {$3$};
\node[circle,draw,inner sep=0pt,minimum width=5mm] (4) at (4.5,3) {$4$};
\node[circle,draw,inner sep=0pt,minimum width=5mm] (5) at (0,0) {$5$};
\node[circle,draw,inner sep=0pt,minimum width=5mm] (6) at (1.5,0) {$6$};
\node[circle,draw,inner sep=0pt,minimum width=5mm] (7) at (3,0) {$7$};
\node[circle,draw,inner sep=0pt,minimum width=5mm] (8) at (4.5,0) {$8$};
\node[circle,draw,inner sep=0pt,minimum width=5mm] (9) at (7,1.5) {$9$};
\draw (1) -- (6);
\draw (1) -- (7);
\draw (1) -- (8);
\draw (2) -- (5);
\draw (2) -- (7);
\draw (2) -- (8);
\draw (3) -- (5);
\draw (3) -- (6);
\draw (3) -- (8);
\draw (4) -- (5);
\draw (4) -- (6);
\draw (4) -- (7);
\draw (1) -- (9);
\draw (2) -- (9);
\draw (3) -- (9);
\draw (4) -- (9);
\draw (5) -- (9);
\draw (6) -- (9);
\draw (7) -- (9);
\draw (8) -- (9);

\end{tikzpicture}

\end{center}
\vspace{-2mm}
\caption{The graph $G_4$ with representation number 4}\label{G4-fig}
\end{figure}


It is unknown whether there exist graphs on $n$ vertices with representation number between $\lfloor n/2 \rfloor$ and essentially $2n$ (the known upper bound), but one has the following conjecture.  

\begin{conjecture}\label{conj2} $G_n$ has the highest representation number among all graphs on $2n+1$ vertices. \end{conjecture}

It is easy to see that $G_n$ is a comparability graph (just make the apex to be a source, or a sink, and orient the remaining crown graph from one part to the other).  Surprisingly, the following result on $G_n$ does not seem to be recorded in the literature. 

\begin{theorem} $G_n$ has the highest representation number among all comparability graphs on $2n+1$ vertices. \end{theorem}

\begin{proof}

Let $G$ be a comparability graph on $2n+1$ vertices. By Theorem~\ref{Kit-Sei-thm}, $G$ can be represented by a concatenation of permutations, which is equivalent to representing the partially ordered set corresponding to $G$ by intersection of linear orders.
It is known \cite{H51} that for any finite poset $P$, the dimension of $P$ is at most half of the number of elements in $P$. Thus, the number of permutations required to represent $G$ cannot exceed $n$, which in turn implies that $\mathcal{R}(G)\leq n$ (dropping the requirement to represent $G$ permutationally, we can only shorten a word-representant). Thus, by Theorem~\ref{G-represent-thm}, $\mathcal{R}(G)\leq\mathcal{R}(G_n)$. \end{proof}

\subsection{Organization of the paper}

Our concern in this paper is word-representation of {\em connected}  graphs, because a graph is word-representable if and only if each of its connected components is word-representable \cite{KL15}. In Section~\ref{uniform-wr-sec} we explain our computational approach using satisfiability module theories (SMT) to study $k$-word-representable graphs and present the results obtained. In particular, we raise some concerns about Conjecture~\ref{conj2}, while confirming it for graphs on at most 9 vertices. 
In Section~\ref{nwr-sec} we present a complementary computational approach using constraint programming, enabling us count connected non-word-representable graphs. In particular, in Section~\ref{nwr-sec} we report that using 3 years of CPU time, we found out that 64.65\% of all connected graphs on 11 vertices are non-word-representable. Another important corollary of our results in Section~\ref{nwr-sec} is the correction of the published result \cite{K17,KL15} on the number of connected non-word-representable graphs on 9 vertices (see Table~\ref{table-constraint}).  
In Section~\ref{sec-3-semi-trans} we introduce the notion of a $k$-semi-transitive orientation refining the notion of a semi-transitive orientation, and show that 3-semi-transitively orientable graphs are not necessarily semi-transitively orientable. 
Finally, in Section~\ref{conc-sec} we suggest a few directions for further research and experimentation.

\section{Finding word-representants by SMT}\label{uniform-wr-sec}

\label{SMT-semi-transitive-section-if-not-merged}

How to find a $k$-uniform word-representation of a given graph $G = (V,E)$?
In this section we discuss how this can be done by means of {\em SMT}: {\em satisfiability modulo theories}. In particular, we focus on the {\em theory of linear inequalities}, and want to exploit
the fact that current SMT solvers are strong in establishing whether a Boolean formula composed from $\wedge$, $\vee$, $\neg$ and linear inequalities admits a solution, and if so, finds one. Here by a solution we mean a choice for the values of the variables such that the formula yields true; if such a solution exists the formula is called `{\em satisfiable}', and the solution is called a `{\em satisfying assignment}'. So, our goal is to find such a Boolean formula for which any solution corresponds to a $k$-uniform word-representation of a given graph. For doing so, we need a way to express the unknown $k$-uniform word of length $k n$, where $n=\# V$ is the number of vertices in the graph in question, by a number of variables.
This is done as follows.
Number the vertices from 1 to $n$, and represent a word $w$ that we are looking for by $kn$ integer variables
$A_{i,j}$, for $i=1,\ldots,n$, $j=1,\ldots,k$. The intended meaning of
$A_{i,j}$ is the position of the $j$-th occurrence of symbol $i$ in $w$, for $i=1,\ldots,n$, $j=1,\ldots,k$. For example,  for the following graph

\vspace{3mm}

\begin{center}
\begin{tikzpicture}
\node[circle,draw,inner sep=0pt,minimum width=5mm] (1) at (0,0) {$1$};
\node[circle,draw,inner sep=0pt,minimum width=5mm] (2) at (2,0) {$2$};
\node[circle,draw,inner sep=0pt,minimum width=5mm] (3) at (4,0) {$3$};
\draw (1) -- (2);
\draw (2) -- (3);
\end{tikzpicture}
\end{center}

\vspace{3mm}

\noindent the word $w = 132312$ is a 2-uniform word-representing the graph, and is expressed by the values
$A_{1,1} = 1$, $A_{1,2} = 5$, $A_{2,1} = 3$, $A_{2,2} = 6$, $A_{3,1} = 2$, $A_{3,2} = 4$.

Now, our formula is the conjunction of a number of requirements on these integer variables $A_{i,j}$ that all together describe a word $w$ representing a given graph $G = (V,E)$. These requirements are:

\begin{itemize}
\item $A_{i,j}>0$, for all $i=1,\ldots,n$, $j=1,\ldots,k$;
\item $A_{i,j} \leq kn$, for all $i=1,\ldots,n$, $j=1,\ldots,k$;
\item all $A_{i,j}$ are distinct (distinctness is a feature included in SMT format);
\item for all $i_1 i_2 \in E$,
\[ (A_{i_1,1} < A_{i_2,1} < A_{i_1,2} < A_{i_2,2} < \cdots < A_{i_1,j} < A_{i_2,j})\]
\[ \vee \; (A_{i_2,1} < A_{i_1,1} < A_{i_2,2} < A_{i_1,2} < \cdots < A_{i_2,j} < A_{i_1,j}); \]
\item for all $i_1 i_2 \not\in E$,
\[ \neg (A_{i_1,1} < A_{i_2,1} < A_{i_1,2} < A_{i_2,2} < \cdots < A_{i_1,j} < A_{i_2,j})\]
\[ \wedge \; \neg (A_{i_2,1} < A_{i_1,1} < A_{i_2,2} < A_{i_1,2} < \cdots < A_{i_2,j} < A_{i_1,j}). \].
\end{itemize}

So, for our graph above, the formula reads
\[ A_{1,1} > 0 \wedge A_{1,2} > 0 \wedge A_{2,1} > 0 \wedge A_{2,2} > 0 \wedge A_{3,1}  > 0 \wedge A_{3,2} > 0 \; \wedge \]
\[ A_{1,1} \leq 6 \wedge A_{1,2} \leq 6 \wedge A_{2,1} \leq 6 \wedge A_{2,2} \leq 6 \wedge A_{3,1}  \leq 6 \wedge A_{3,2} \leq 6 \; \wedge \]
\[ \mbox{distinct}(A_{1,1},A_{1,2},A_{2,1},A_{2,2},A_{3,1},A_{3,2}) \; \wedge \]
\[ ((A_{1,1} < A_{2,1} < A_{1,2} < A_{2,2}) \vee (A_{2,1} < A_{1,1} < A_{2,2} < A_{1,2})) \; \wedge \]
\[ ((A_{2,1} < A_{3,1} < A_{2,2} < A_{3,2}) \vee (A_{3,1} < A_{2,1} < A_{3,2} < A_{2,2})) \; \wedge \]
\[ \neg (A_{1,1} < A_{3,1} < A_{1,2} < A_{3,2}) \wedge \neg (A_{3,1} < A_{1,1} < A_{3,2} < A_{1,2})). \]
For the values $A_{1,1} = 1$, $A_{1,2} = 5$, $A_{2,1} = 3$, $A_{2,2} = 6$, $A_{3,1} = 2$, $A_{3,2} = 4$ this formula yields true, as is found
by the SMT solver Z3, yielding the 2-uniform word-representation $w = 132312$ of the graph.

Up to syntactic details (boolean operators are written as `not', `and', `or', all operators are written in prefix notation), it is exactly this formula on which
an SMT solver like Z3 \cite{z3} or YICES \cite{yices} can be applied, yielding `satisfiable', and the corresponding satisfying assignment gives our values of $A_{i,j}$.

We wrote a tool doing this in a way where the internal use of an SMT solver is hidden for the user. It is available on 

\centerline{{\tt http://www.win.tue.nl/}$\sim${\tt hzantema/reprnr.html}.}

\noindent The tool reads a graph and then tries to find a $k$-representation for $k=2,3,4,\ldots$ by building the formula as presented above and then calling an SMT solver. As soon as a satisfying
assignment is found, the computation stops and the resulting values are transformed to the corresponding $k$-uniform word-representation, which is returned to the user. The tool is available both for Windows (calling the SMT solver Z3) and for Linux (calling the SMT solver YICES), together with several examples. Typically, for graphs like the cube, the prism on the triangle, Petersen graph, and $G_4$ (see below), the $k$-uniform word representing the graph is found in a second or less.

\begin{table}
\begin{center}
\begin{tabular}{|c|r|r|r|r|r|r|}
\hline
\# of & \# of conn. & \multicolumn{5}{c|}{representation number} \\ \cline{3-7}
vertices & graphs & 1 & 2 & 3 & 4 & $> 4$ \\ \hline \hline
3 & 2 & 1 & 1 & 0 & 0 & 0 \\ \hline
4 & 6 & 1 & 5 & 0 & 0 & 0 \\ \hline
5 & 21 & 1 & 20 & 0 & 0 & 0 \\ \hline
6 & 112 & 1 & 109 & 1 & 0 & 1 \\ \hline
7 & 853 & 1 & 788 & 39 & 0 & 25 \\ \hline
8 & 11,117 & 1 & 8335 & 1852 & 0 & 929 \\ \hline
9 & 261,080 & 1 & 117,282 & 88,838 & 2 & 54,957 \\ \hline
\end{tabular}
\end{center}
\caption{Distribution of $k$-representable graphs on at most 9 vertices}\label{tab-k-repr}
\end{table}

As this tool works quite quickly, it is feasible to run it on a great number of graphs. In particular, we ran it on all connected graphs on $\leq 9$ vertices as they are available from \\ \centerline {{\tt http://users.cecs.anu.edu.au/}$\sim${\tt bdm/data/graphs.html.}} \\ The results are listed in Table~\ref{tab-k-repr}, where `representation number $> 4$' means that no 4-representation exists, so either the representation number is $> 4$, or the graph is not word-representable (for which the representation number is $\infty$). However, as these numbers coincide with the respective numbers in Table~\ref{table-constraint}, we conclude that only the latter occurs, and no word-representable graph exists on $\leq 9$ vertices with representation number $> 4$.

\begin{figure}
\begin{center}
\vspace{3mm}
\noindent\includegraphics[scale=0.27]{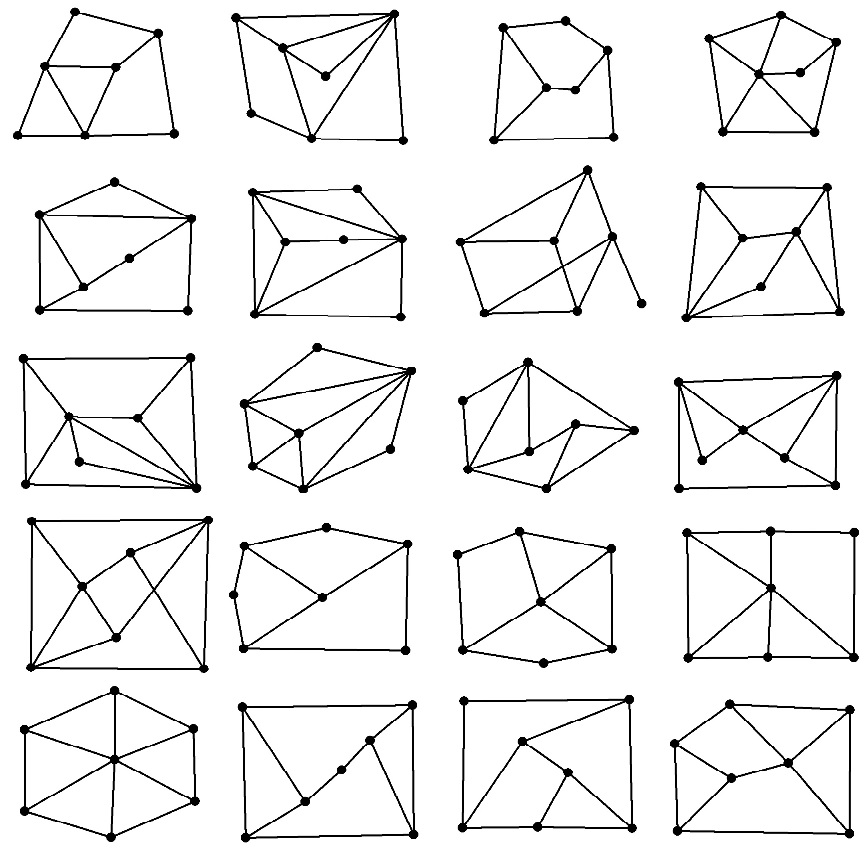} 
\includegraphics[scale=0.27]{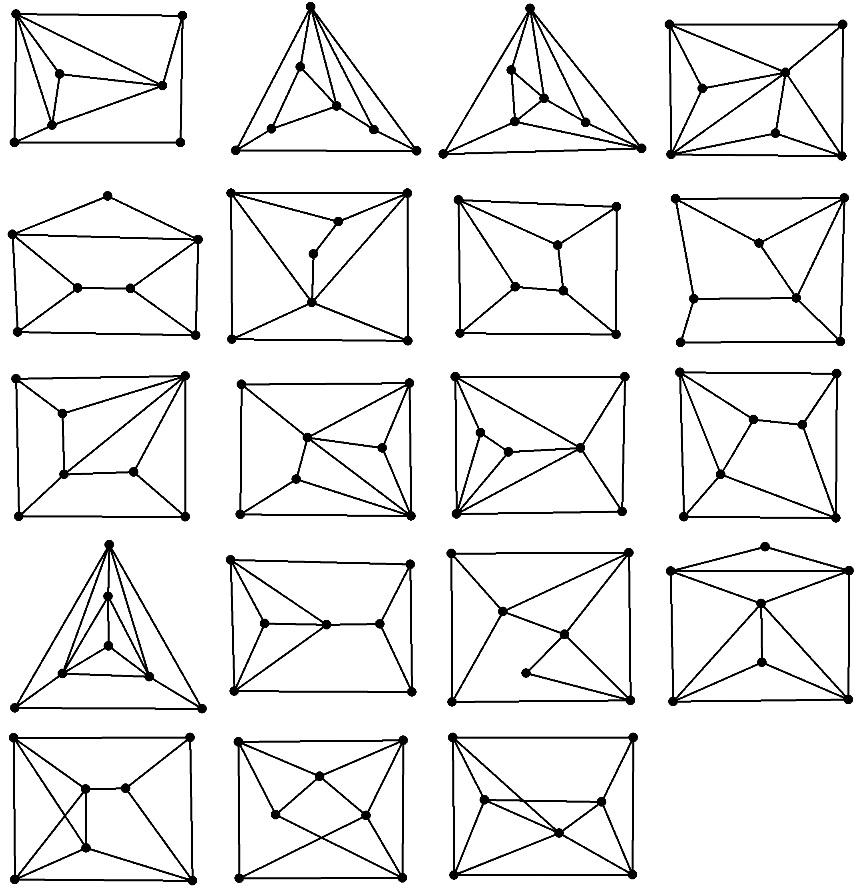} 
\vspace{3mm}
\end{center}
\vspace{-2mm}
\caption{The 39 graphs on 7 vertices with representation number 3 }\label{39-7-3}
\end{figure}

The single graph on 6 vertices with representation number 3 is the prism on the triangle; the single non-word-representable graph on 6 vertices is the wheel on 5 vertices.
The 39 graphs on 7 vertices with representation number 3 are given in Figure~\ref{39-7-3}.

The most surprising result in Table~\ref{tab-k-repr} is the two graphs on 9 vertices with representation number 4. One of them was known before, namely,
$G_4$ presented in Figure~\ref{G4-fig}, and it was believed to be the only graph on 9 vertices with representation number 4.  However, our computations have shown the existence of another such graph, namely the graph $J_4$ shown in Figure~\ref{9-4}. We note that $J_4$ is a non-comparability graph, which is easy to check, while $G_4$ is. This may suggest that Conjecture~\ref{conj2} might not be true, since there are many more non-comparability graphs than comparability graphs, and one may expect finding those of them that have higher representation number than $G_n$. Having said that, we were not able to extend the construction of $J_4$ (in a natural way) to more than 9 vertices. 


\begin{figure}
\begin{center}

\begin{tikzpicture}
\node[circle,draw,inner sep=0pt,minimum width=5mm] (1) at (2.5,4.6) {$1$};
\node[circle,draw,inner sep=0pt,minimum width=5mm] (2) at (0,0) {$2$};
\node[circle,draw,inner sep=0pt,minimum width=5mm] (3) at (5,0) {$3$};
\node[circle,draw,inner sep=0pt,minimum width=5mm] (4) at (2.5,2.7) {$4$};
\node[circle,draw,inner sep=0pt,minimum width=5mm] (5) at (1.5,1) {$5$};
\node[circle,draw,inner sep=0pt,minimum width=5mm] (6) at (3.5,1) {$6$};
\node[circle,draw,inner sep=0pt,minimum width=5mm] (7) at (1.5,2) {$7$};
\node[circle,draw,inner sep=0pt,minimum width=5mm] (8) at (3.5,2) {$8$};
\node[circle,draw,inner sep=0pt,minimum width=5mm] (9) at (2.5,0.5) {$9$};
\draw (1) -- (2);
\draw (1) -- (3);
\draw (2) -- (3);
\draw (4) -- (5);
\draw (4) -- (6);
\draw (5) -- (6);
\draw (1) -- (7);
\draw (2) -- (7);
\draw (4) -- (7);
\draw (5) -- (7);
\draw (1) -- (8);
\draw (3) -- (8);
\draw (4) -- (8);
\draw (6) -- (8);
\draw (2) -- (9);
\draw (3) -- (9);
\draw (5) -- (9);
\draw (6) -- (9);
\end{tikzpicture} \hspace{1cm}
\begin{tikzpicture}
\node[circle,draw,inner sep=0pt,minimum width=5mm] (1) at (0,2.5) {$1$};
\node[circle,draw,inner sep=0pt,minimum width=5mm] (2) at (1,3) {$2$};
\node[circle,draw,inner sep=0pt,minimum width=5mm] (3) at (1,2) {$3$};
\node[circle,draw,inner sep=0pt,minimum width=5mm] (4) at (4,2.5) {$4$};
\node[circle,draw,inner sep=0pt,minimum width=5mm] (5) at (3,3) {$5$};
\node[circle,draw,inner sep=0pt,minimum width=5mm] (6) at (3,2) {$6$};
\node[circle,draw,inner sep=0pt,minimum width=5mm] (7) at (2,5) {$7$};
\node[circle,draw,inner sep=0pt,minimum width=5mm] (8) at (2,0) {$8$};
\node[circle,draw,inner sep=0pt,minimum width=5mm] (9) at (2,2.5) {$9$};
\draw (1) -- (2);
\draw (1) -- (3);
\draw (2) -- (3);
\draw (4) -- (5);
\draw (4) -- (6);
\draw (5) -- (6);
\draw (1) -- (7);
\draw (2) -- (7);
\draw (4) -- (7);
\draw (5) -- (7);
\draw (1) -- (8);
\draw (3) -- (8);
\draw (4) -- (8);
\draw (6) -- (8);
\draw (2) -- (9);
\draw (3) -- (9);
\draw (5) -- (9);
\draw (6) -- (9);
\end{tikzpicture}

\end{center}
\vspace{-2mm}
\caption{The graph $J_4$ with representation number 4. It is shown in two ways to show different symmetries.}\label{9-4}
\end{figure}
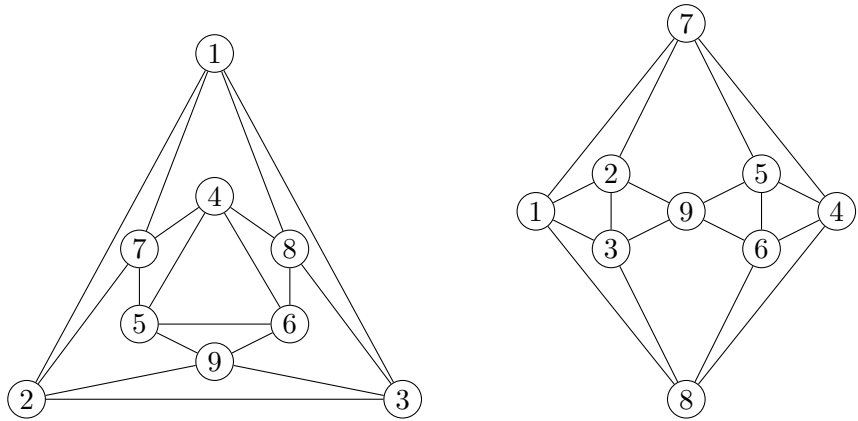

\section{Counting non-word-representable graphs using constraint programming}\label{nwr-sec}

Similarly to our studies of $k$-word-representable graphs, we performed large computations using constraint programming \cite{constraints-handbook} to count the numbers of non-word-representable connected graphs with up to 11 vertices.  To do this, we used the constraint modelling tool Savile Row \cite{Nightingale2017:automatically} and  the constraint solver Minion \cite{Gent2006:minion}.  These tools have been used successfully in the past to obtain novel enumerations of a variety of combinatorial structures including semigroups \cite{semigroups-10}, equidistant frequency permutation arrays \cite{efpa}, and S-crucial and bicrucial permutations
with respect to squares \cite{squarefree}.  

 Our starting point was to model the concept of word-representability in a way similar to that when using SMT in Section~\ref{SMT-semi-transitive-section-if-not-merged}.  However, here we use Theorem~\ref{key-thm} showing the equivalence between word-representability and semi-transitivity, so that semi-transitive orientations are now used to determine whether or not a graph is word-representable.  As with SMT in Section~\ref{sec-3-semi-trans}, we use a boolean $u_{ij}$ to indicate an undirected edge between $i$ and $j$, and a boolean $e_{ij}$ to indicate a directed edge from $i$ to $j$. Moreover, we use a boolean $t_{ij}$ to indicate the {\em transitive closure} of $e$, which is true when there is a path of directed edges from $i$ to $j$.  Most of the model expresses the appropriate linkages between these sets of variables.  For example, constraints in the model express that $t$ is the transitive closure of $e$.  The acyclicity of $e$ is elegantly expressed by each $t_{ii}$ being false,  i.e.\ no vertex being reachable from itself in the transitive closure.  The final constraint expresses the property of semi-transitivity.  This states that if two vertices are connected by a directed path, and there is an undirected edge between them, then all pairs of intermediate vertices must have a directed edge between them in the appropriate direction.  
 
The model we used is shown in full in Figure~\ref{eprime-wrg}.  There are three points of detail about the model which deserve mention.  
First, this model neither  check graphs for being connected, nor for being non-isomorphic to each other. 
This is not easy to do very efficiently in constraints, so instead we constructed a list of all connected undirected graphs with no two graphs being isomorphic, using the program \texttt{geng} \cite{geng}. 
Second, we originally modelled an undirected graph as an input to the constraint model, which was then checked for word-representability.  However, this proved to be very inefficient as the vast majority of the constraint modelling processes was the same for each graph.  Instead, we provide the constraint model with a list of graphs produced by \texttt{geng} and insist that the solution is one of those graphs. This is achieved in constraints using the `table' constraint, which can be propagated very efficiently \cite{DBLP:reference/fai/Bessiere06}.  As well as saving work at the modelling stage, it also provides the capability to save work at the solving stage. For example, if all graphs remaining for consideration contain a certain undirected edge $ij$, the  variable $u_{ij}$ can be set true immediately.  A major advantage of this approach is that it makes it particularly easy to parallelise the enumeration process, simply by splitting the list of distinct connected graphs into appropriately sized chunks.
Finally, the line `\texttt{branching on [u]}' tells the constraint tools that we only wish to solve the problem once for each different assignment of \texttt{u}, i.e. for each undirected graph.  Without this, any graph admitting more than one semi-transitive orientation would be repeated in the output, wasting both search time and necessitating extra work in removing duplicates.  

Results of our computations are shown in Table~\ref{table-constraint}. Note that in one case numbers are different to those previously reported. 
The true number of connected non-word-representable graphs on 9 vertices is 54,957, {\em not} 68,545 as was reported in \cite{K17,KL15} (which was a copy/paste mistake).

It is also interesting to identify minimal non-word-representable graphs of each size, i.e.\ graphs containing no non-word-representable strict induced subgraphs.  
To do this, we stored all non-word-representable graphs of each size. After computing with \texttt{geng} all possible graphs with one more vertex, we eliminate graphs containing one of the stored graphs as an induced subgraph.  
We did this with a simple constraint model which tries to find a mapping from the vertices of the induced subgraph to the vertices of the larger graph, and if successful discards the larger graph from consideration.
This enabled us to count all minimal non-word-representable graphs of each size up to $9$, which is shown in Table~\ref{table-constraint}. The filtering process we used was too inefficient to complete the cases  $n\geq 10$.

\begin{figure}
{\small
\begin{verbatim}
language ESSENCE' 1.0
given n : int              
given triangle_table : matrix indexed by [int(1..numgraphs),int(1..(n-1)*(n)/2)] 
                       of int(0,1)
letting LETTER be domain int(1..n)

find upper_triangle : matrix indexed by [int(1..((n-1)*n/2))] of int(0,1)
find u : matrix indexed by [LETTER, LETTER] of int(0,1) $ graph undirected edges
find e : matrix indexed by [LETTER, LETTER] of int(0,1) $ graph directed edges
find t : matrix indexed by [LETTER, LETTER] of int(0,1) $ transitive closure 
branching on [u]

such that
    $ the diagonal is empty
    forAll i : LETTER . u[i,i] = 0,
    $ the graph is undirected
    forAll i,j : LETTER . u[i,j] = u[j,i],
    
    $ linking u and the upper triangle
    forAll i,j : LETTER . i < j -> (u[i,j] = upper_triangle[n*(i-1)+j-((i+1)*i/2)]),
    $ the graph is one of the preprocessed graphs
    table(upper_triangle,triangle_table),

    $ linking e and u
    forAll i,j : LETTER . u[i,j] = 0 -> e[i,j]=0,
    forAll i,j : LETTER . u[i,j] = 1 -> ((e[i,j]=1) \/ e[j,i]=1),

    $ directed graph is irreflexive and antisymmetric
    forAll i : LETTER . e[i,i] = 0,
    forAll i,j : LETTER . i < j -> ( (e[i,j] = 0) \/ (e[j,i] = 0)),

    $ t is transitive closure of e and is acyclic
    forAll i,j : LETTER . (e[i,j] = 1) -> (t[i,j] = 1),    
    forAll i,j,k : LETTER . ( (t[i,j] = 1) /\ (t[j,k] = 1)) -> (t[i,k] = 1),
    forAll i : LETTER . t[i,i] = 0,

    $ semi transitive ordering    
    forAll i,k: LETTER . 
        ((t[i,k] = 1) /\ (u[i,k] = 1)) -> 
      	     ((ordering[i,k] = 1) /\
              forAll j : LETTER . 
                   ((t[i,j]=1 /\ t[j,k]=1) ->  (e[i,j] = 1 /\ e[j,k]=1)))
\end{verbatim}
}
\caption{Essence Prime model of word-representable graphs}
\label{eprime-wrg}
\end{figure}

\begin{table}
\begin{center}
\begin{tabular}{|r|r|r|r|r|r|r|}
\hline
\# of & \# of conn. & \multicolumn{5}{c|}{All non-word-representable graphs} \\ \cline{3-7}
vert. & graphs & Total & \% of cand.        &  Time   & Min. & Non-Min.    \\ \hline \hline
6   &           112 &           1 &  0.89\% &   3.0s &   1 &      0  \\ \hline
7   &           853 &          25 &  2.93\% &   4.0s &  10 &     15  \\ \hline
8   &        11,117 &         929 &  8.36\% &    26s &  47 &    882  \\ \hline
9   &       261,080 &      54,957 & 21.05\% &    29m & 179 & 54,778  \\ \hline
10  &    11,716,571 &   4,880,093 & 41.65\% &    74h &  -  &  -      \\ \hline
11  & 1,006,690,565 & 650,856,040 & 64.65\% & 1,100d &  -  &  -      \\ \hline
\end{tabular}
\caption{The numbers of all non-word-representable graphs, as well as the numbers of such graphs, called {\em non-minimal}, that include smaller non-word-representable subgraphs, and those, called {\em minimal}, that do not. The percentage of non-word-representable graphs to all graphs is given to 2 decimal places. 
Times indicate the CPU time used to compute all non-word-representable graphs, to 2 significant figures in an appropriate unit (seconds, minutes, hours, days). The time to count minimal/non-minimal graphs is not shown.}
\label{table-constraint}
\end{center}

\end{table}

\section{Refining semi-transitivity}\label{sec-3-semi-trans}

The notion of $k$-word-representability refines that of word-representability. However, Theorem~\ref{equiv-thm} shows that these notions are equivalent. Still, $k$-word-representability plays a very important role in the theory of word-representable graphs. 

Thinking along similar lines, we introduce the potentially useful notion of a {\em $k$-semi-transitive orientation} refining semi-transitive orientations linked to word-representability via Theorem~\ref{key-thm}.  Recall the definition of a shortcut in Section~\ref{semi-transitive-sec}. An undirected graph is $k$-semi-transitively oriented, or {\em $k$-semi-transitive} for brevity, if it admits an acyclic orientation avoiding shortcuts of length $k$ (longer shortcuts are allowed). In particular, an undirected graph is 3-semi-transitive if it admits an acyclic orientation such that for any directed path $v_0\rightarrow v_1\rightarrow v_2\rightarrow v_3$ of length 3 for which $v_0\rightarrow v_3$ is an edge, also
$v_0 \rightarrow v_2$ and $v_1\rightarrow v_3$ are edges.

The notion of 3-semi-transitivity is easily expressed in SMT. Writing $u_{ij}$ for the boolean expressing whether there is an undirected edge from $i$ to $j$,
and $e_{ij}$ for the boolean expressing whether there is a directed edge from $i$ to $j$, the connection between directed and undirected graph is expressed by
\[ u_{ij} \Leftrightarrow (e_{ij} \vee e_{ji}) \]
for all vertices $i,j$. Being acyclic is expressed by the existence of a weight function $w$ such that
\[ e_{ij} \Rightarrow w(i) > w(j) \]
for all vertices $i,j$. Finally, the path condition is expressed by
\[ ((\exists k,m : ((e_{ik} \wedge e_{jk}) \vee (e_{ki} \wedge e_{kj})) \wedge e_{im} \wedge e_{mj}) \; \Rightarrow \; e_{ij}\]
for all vertices $i,j$, where $\exists$ runs over the vertices.
For a given undirected graph, we take the conjunction of the above requirements and for all $i,j$ we add $\wedge u_{ij}$ if there is an edge from $i$ to $j$, and add $\wedge \neg u_{ij}$ otherwise. Then, by construction, the resulting formula is satisfiable if and only if the undirected graph is 3-semi-transitive. We built these formulas for all connected graphs on $\leq 9$ vertices, and applied Z3 on them. As a result, we determined that for $\leq 8$ vertices a graph is 3-semi-transitive if and only if it is word-representable. In contrast, for 9 vertices we determined that there are exactly 4 graphs that are 3-semi-transitive but not word-representable, and hence not semi-transitive. They are depicted in Figure~\ref{3-semi-tr-ori-graphs}.   An SMT encoding of checking semi-transitivity is also included in the tool linked to in Section~\ref{uniform-wr-sec}.

Using a similar encoding of the problem, these computational results were extended to finding the number of all 3-semi-transitively orientable graphs on up to 10 vertices using the constraint programming methods described in Section~\ref{nwr-sec}. We refer to Table~\ref{table-constraint2} where these results are recorded along with the number of minimal (not containing smaller such graphs as induced subgraphs) non-3-semi-transitively orientable graphs. Comparing Tables~\ref{table-constraint} and~\ref{table-constraint2}, we see that there are 585 3-semi-transitively orientable, but not semi-transitively orientable graphs on 10 vertices. 

\begin{figure}
\begin{center}
\includegraphics[scale=0.006]{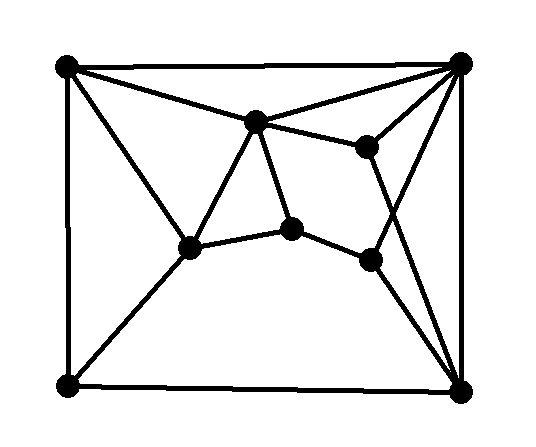}
\includegraphics[scale=0.006]{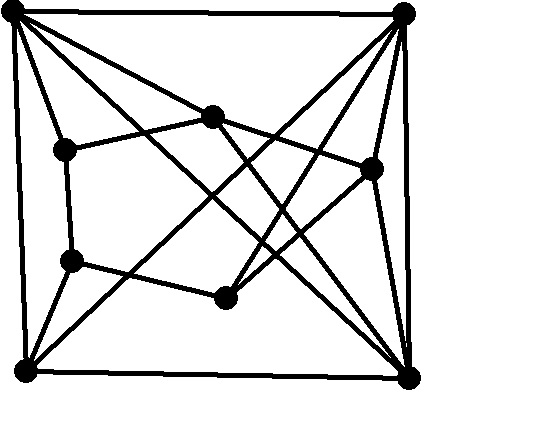}
\includegraphics[scale=0.006]{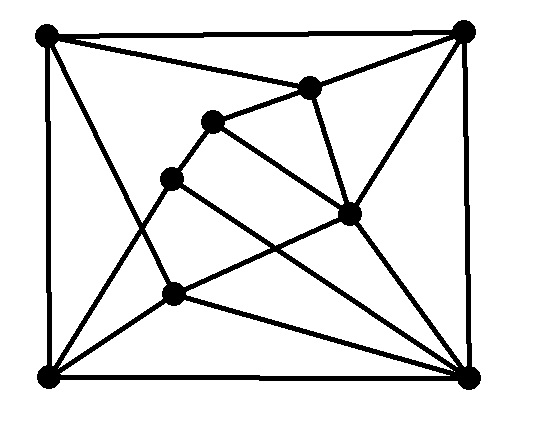}
\includegraphics[scale=0.006]{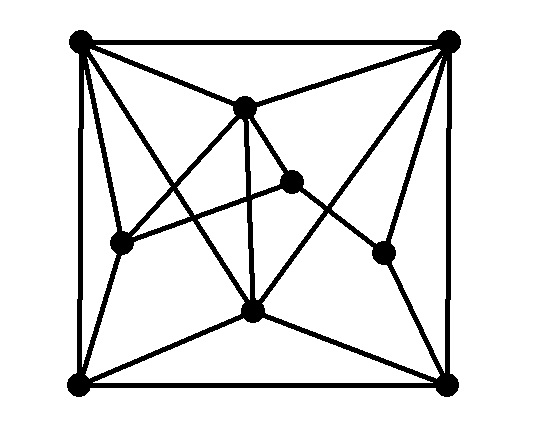}
\end{center}
\vspace{-2mm}
\caption{3-semi-transitively, but not semi-transitively orientable graphs}\label{3-semi-tr-ori-graphs}
\end{figure}

\begin{table}
\begin{center}
\begin{tabular}{|r|r|r|r|r|r|r|}
\hline
\# of & \# of conn. & \multicolumn{5}{c|}{All non-3-semi-transitively orientable graphs} \\ \cline{3-7}
vert. & graphs & Total & \% of cand.        &  Time   & Minimal & Non-Minimal    \\ \hline \hline
6   &           112 &           1 &  0.89\% &   4.0s &   1 &      0  \\ \hline
7   &           853 &          25 &  2.93\% &   6.0s &  10 &     15  \\ \hline
8   &        11,117 &         929 &  8.36\% &    80s &  47 &    882  \\ \hline
9   &       261,080 &      54,953 & 21.05\% &   2.8h & 175 & 54,778  \\ \hline
10  &    11,716,571 &   4,879,508 & 41.65\% &    22d &  -  &  -      \\ \hline
\end{tabular}
\caption{Numbers of (minimal) non-3-semi-transitively orientable graphs and the CPU time to obtain them. The time to count minimal/non-minimal graphs is not shown.
}

\label{table-constraint2}
\end{center}

\end{table}

Thus, the notions of $k$-semi-transitively orientable graphs and semi-transitively orientable graphs are not equivalent.

\section{Concluding remarks}\label{conc-sec}

We conclude by suggesting a few directions of further research relevant to our paper. In each of these directions one can use the computational approaches/tools developed by us  to support finding new results. In particular, one could try to use our tools to take all bipartite graphs and to test Conjecture~\ref{conj1} for larger graphs.

It would be interesting to extend the construction of $J_4$ in Figure~\ref{9-4} (in a natural way) to more than 9 vertices so that new graphs with high representation numbers would be obtained. This may help to prove or disprove Conjecture~\ref{conj2}. 

Also, an intriguing question is whether or not there exists $k$ such that semi-transitive orientability is equivalent to $k$-semi-transitively orientability. If such a $k$ exists, it must be $>3$ (e.g.\ because of the graphs in Figure~\ref{3-semi-tr-ori-graphs}). In either case, to study  the properties of $k$-semi-transitively orientable graphs (at least 3-semi-transitively orientable graphs) is an interesting and challenging direction of research. Many questions that can be asked about word-representable graphs \cite{K17,KL15} can be asked about $k$-semi-transitively orientable graphs, e.g.\ how many such graphs there are, or how we can describe these graphs in terms of forbidden subgraphs, etc, etc.

Finally, even though it seems that our current methods would not be able to extend the results of Table~\ref{table-constraint} to 12 vertices, it is interesting if it would be ever  possible to achieve.

\end{document}